\documentclass[a4paper]{amsart}

\usepackage{amsmath,amssymb}
\usepackage{amsthm}
\usepackage{mathtools}
\usepackage[capitalize]{cleveref}
\usepackage{todonotes}
\usepackage{tikz-cd}
\usepackage{quiver}
\usepackage{lmodern}
\usepackage[mathscr]{eucal}
\usepackage[T1]{fontenc}
\usepackage{slashed}
\usepackage{comment}

\makeatletter
\def\DefineSymbol#1#2{\newcommand{#1}{{\mathrm{#2}}}}
\def\DefineCategory#1#2{\newcommand{#1}{{\mathsf{#2}}}}
\theoremstyle{plain}
    \newtheorem{theorem}{Theorem}[section]
    \newtheorem{lemma}[theorem]{Lemma}
    
    \newtheorem{corollary}[theorem]{Corollary}
    
\theoremstyle{definition}
    \newtheorem{definition}[theorem]{Definition}
    
\theoremstyle{remark}
    \newtheorem{remark}[theorem]{Remark}
    \newtheorem{example}[theorem]{Example}
    \numberwithin{equation}{section}
\crefname{lemma}{Lemma}{Lemma}
\DefineSymbol{\pt}{pt}
\DefineSymbol{\grouplike}{gp}
\DefineSymbol{\id}{id}
\DefineSymbol{\proj}{proj}
\DefineSymbol{\incl}{incl}
\DefineSymbol{\const}{const}
\DefineSymbol{\op}{op}
\DefineSymbol{\inj}{inj}
\let\proj\relax
\DefineSymbol{\proj}{proj}

\DeclareMathOperator{\Aut}{Aut}
\DeclareMathOperator{\SWF}{SWF}
\DeclareMathOperator{\Br}{Br}
\DeclareMathOperator{\Map}{Map}
\DeclareMathOperator*{\Mapl}{Map}

\DeclareMathOperator*{\hocolim}{hocolim}
\DeclareMathOperator*{\colim}{colim}

\DeclareMathOperator{\Ind}{Ind}
\DefineCategory{\Set}{Set}
\DefineCategory{\sset}{sSet}
\DefineCategory{\msset}{sSet^+}
\DefineCategory{\Ab}{Ab}
\DefineCategory{\Alg}{Alg}
\DefineCategory{\Mod}{Mod}
\DefineCategory{\Sp}{Sp}
\DefineCategory{\Prl}{Pr^L}
\DefineCategory{\Fun}{Fun}
\DefineCategory{\Orb}{Orb}
\DefineCategory{\Pre}{Pre}
\DefineCategory{\Top}{Top}
\DefineCategory{\rfib}{RFib}
\DefineCategory{\base}{\Omega}
\DefineCategory{\calg}{CAlg}
\DefineCategory{\PrlG}{Pr^L_{\mathit{G}}}
\DefineCategory{\CatG}{Cat_{\mathit{G}}}

\newcommand{\unit}{\mathbf{1}}
\renewcommand{\S}{{\mathbb{S}}}
\newcommand{\GL}{{\mathit{GL}}}
\newcommand{\B}{\mathit{B}}

\newcommand{\Cat}{{\mathsf{Cat}_{\infty}}}
\newcommand{\lCat}{{\hat{\mathsf{Cat}}_{\infty}}}
\newcommand{\lFun}{{\mathsf{Fun}^{\mathrm{colim}}}}

\newcommand{\Pic}{\mathrm{Pic}}
\newcommand{\lCatG}{{\hat{\CatG}}}
\newcommand{\Vect}{\mathsf{Vect}_\Cbb}

\newcommand{\Abb}{\mathbb{A}}

\newcommand{\Cbb}{\mathbb{C}}

\newcommand{\Ebb}{\mathbb{E}}

\newcommand{\Rbb}{\mathbb{R}}

\newcommand{\Zbb}{\mathbb{Z}}

\newcommand{\Hcal}{\mathcal{H}}

\newcommand{\Tcal}{\mathcal{T}}

\newcommand{\Cscr}{\mathscr{C}}
\newcommand{\Dscr}{\mathscr{D}}
\newcommand{\Escr}{\mathscr{E}}

\newcommand{\Oscr}{\mathscr{O}}

\newcommand{\Rscr}{\mathscr{R}}
\newcommand{\Sscr}{\mathscr{S}}

\renewcommand{\hat}{\widehat}

\newcommand{\setmid}{\mathrel{}\middle|\mathrel{}}

\newcommand{\coend}[2][]{\smashoperator[r]{\int^{#1}_{#2}}\;}

\newcommand*{\relrelbarsep}{.386ex}
\newcommand*{\relrelbar}{%
  \mathrel{%
    \mathpalette\@relrelbar\relrelbarsep
  }%
}
\newcommand*{\@relrelbar}[2]{%
  \raise#2\hbox to 0pt{$\m@th#1\relbar$\hss}%
  \lower#2\hbox{$\m@th#1\relbar$}%
}
\providecommand*{\rightrightarrowsfill@}{%
  \arrowfill@\relrelbar\relrelbar\rightrightarrows%
}
\providecommand*{\leftleftarrowsfill@}{%
  \arrowfill@\leftleftarrows\relrelbar\relrelbar%
}
\providecommand*{\xrightrightarrows}[2][]{%
  \ext@arrow 0359\rightrightarrowsfill@{#1}{#2}%
}
\providecommand*{\xleftleftarrows}[2][]{%
  \ext@arrow 3095\leftleftarrowsfill@{#1}{#2}%
}

\makeatother

\title{Even-periodic cohomology theories for twisted parametrized spectra}
\author{Takumi Maegawa}

\begin{document}

\begin{abstract}
    We recall the notion of twisted parametrized spectra defined by Douglas and provide a sufficient condition for an $\infty$-category of twisted parametrized module spectra to be untwisted over an even-periodic $E_2$-ring.
    It is an easy consequence of the universal property of Thom spectra.%

    We also investigate a genuine equivariant generalization based on the theory of $\infty$-categories internal to the $\infty$-topos of $G$-spaces for a compact Lie group $G$.

    We expect that our sufficient condition is satisfied in a number of gauge theoretic settings.
    This article is intended as a warm-up for a generalization of certain Seiberg-Witten Floer stable homotopy theory, which we look forward to.
\end{abstract}

\maketitle
\tableofcontents

\section{Introduction}

The notion of twisted parametrized spectra was introduced by C. Douglas in~\cite{Douglas}.
Such objects occasionally arise in nature. 
For example, for a closed oriented $3$-manifold $Y$ with $b_1=0$, there is a construction of a stable homotopy type $\SWF(Y)$ as in~\cite{b1=0}, called the Seiberg-Witten Floer homotopy type, which recovers a certain flavor of the monopole Floer homology of $Y$ as its equivariant Borel homology, due to Lidman and Manolescu~\cite{FloerHomologies}. 
Furthermore, by applying generalized cohomology theories to the Floer homotopy type, there can be produced several useful invariants of $3$-manifolds, from which the $10/8$-type inequalities for spin $4$-manifolds having boundary $Y$ is obtained in~\cite{spin4}, for example.

The construction of $\SWF(Y)$ involves stabilizations of certain homotopy types by the \textit{finite-dimensional approximations} $V_\lambda^\mu$ of the given infinite-dimensional vector space $V$ on which the Seiberg-Witten equation is defined. Namely, it is roughly of the following form.
\begin{equation}\label{swf}
    \SWF(Y) \simeq \hocolim_{\substack{\lambda \to -\infty \\ \mu \to +\infty}} \Sigma^{-\Cbb^n}\Sigma^{-V^0_\lambda} \Sigma^\infty I_\lambda^\mu\;.
\end{equation}
Here, the approximation $V_\lambda^\mu$ is the subspace spanned by the eigenvectors of a self-adjoint Fredholm operator $\slashed{D}$ defined on $V$ with eigenvalue in the interval $(\lambda,\mu ]$, and the homotopy type $I_\lambda^\mu$ is the so-called Conley index of the Seiberg-Witten flow restricted to $V_\lambda^\mu$.
The assumption $b_1=0$ is inevitable 
when we seek for $\SWF(Y)$ as a usual stable homotopy type, since the Seiberg-Witten flow in question contains a differential operator parametrized by the Picard torus $H^1(Y,\Rbb)/H^1(Y,\Zbb)$ of $Y$.
The situation suggests that in order to construct the Seiberg-Witten Floer homotopy type for an arbitrary closed $3$-manifold $Y$, we can no longer hope $\SWF(Y)$ to be an object in the classical stable homotopy category $\Sp$.
A similar problem occurs when we attempt to construct a parametrized version of the Seiberg-Witten Floer homotopy type, that is, Seiberg-Witten Floer homotopy types for bundles of $3$-manifolds.
In both cases, the major issue is that the vector space $V$ and the eigenvalues of the operator $\slashed{D}$  do depend on a point in some space $X$, hence the finite-dimensional approximation method is only valid locally on $X$.

Taking the domain $V$ of the Seiberg-Witten map to be a vector bundle over $X$, and regarding the eigenvalues of $\slashed{D}$ as sorts of multi-valued functions on $X$ that may contain nontrivial monodromy, the expression~\eqref{swf} is now expected to be a global section of some bundle of categories fibered over $X$ whose fibers are stable homotopy categories.
To model such phenomena, Douglas used the term \textit{twisted parametrized spectra} and investigated elementary properties of them.
Since lots of geometric applications have been deduced from the cohomology of Floer homotopy types, we are especially interested in the existence of certain cohomology groups for those twisted parametrized spectra.

Roughly speaking, a twisted parametrized spectrum is locally a parametrized spectrum, but is patched together by a cocycle valued in the group $\Pic_\S$ of invertible spectra, or the automorphism group of the category $\Sp$ of spectra.
The definition is conveniently well-written in the language of $\infty$-categories as is already done in~\cite{ABG}. 

In~\cite{spin4}, it is observed that the desuspension appearing in~\eqref{swf} can be replaced by the negative suspension by a complex $\mathrm{Pin}(2)$-representation.
Therefore, it seems justifiable to take our mathematical model which focuses on $\Pic_\S$-valued cocycles 
that are coming from cocycles valued in the infinite loop space $\Zbb\times BU$ of virtual complex vector bundles, which reflects the intuitive situation that the Seiberg-Witten Floer homotopy types are locally constructed via desuspension by complex vector spaces.
In this direction, there is a sufficient condition for the existence of cohomology theories for such twisted parametrized spectra. The main result can be stated as follows.

\begin{theorem}\label{main}
    Let $X$ be a CW-complex, $h$ a map $X \to \B\Pic_\S$, and $R$ an even-periodic $\Ebb_2$-ring spectrum. Suppose that the homotopy class of $h$ can be lifted as follows.
    \begin{equation*}
        \begin{tikzcd}
            {} & \B(\mathbb{Z}\times \B U) \arrow[d, "\B J_\Cbb"] \\
            X \arrow[ru, dashed] \arrow[r, "h"'] & \B\mathrm{Pic}_{\mathbb{S}}
        \end{tikzcd}
    \end{equation*}
    Then any twisted parametrized spectrum, that is, an object $\Tcal$ in the $\infty$-category $\lim \left(X \xrightarrow{h} B\Pic_\S \to \lCat\right)$, admits the twisted $R$-cohomology group $R^\bullet(\Tcal)$, or rather the $R$-module homotopy type $R\otimes\Tcal$ in $(\Mod_R){}_{/X}$.
\end{theorem}

It is worth noticing that the space $\B(\Zbb\times\B U)=\Omega^{\infty-1}ku$ is the classifying space for the first $KU$-cohomology group since the group $K^1(X)$ classifies Hilbert bundles $\Hcal$ over $X$ together with self-adjoint Fredholm operators $\slashed{D}\colon \Hcal \to \Hcal$, which in fact naturally arise in the setting of certain gauge theories.
In other words, the ``twisting datum'' $h$ and also the assumption of the theorem is \emph{automatically obtained} once we specify a Hilbert bundle and a certain linear part of a gauge theoretic equation.

\begin{remark}
    There is a universal even-periodic ring spectrum $MUP$ namely the periodic complex bordism spectrum, which is in fact an $\Ebb_\infty$-ring. The theorem in particular states that the $MUP$-module homotopy types parametrized over $X$ are well-defined under the assumption.
    There is also a growing expectation on various Floer homotopy theories that they can generally be constructed over the $\Ebb_\infty$-ring $MUP$, as is partly advertised by Abouzaid and Blumberg in~\cite{abouzaid}.

    A more interesting direction may be the case $R=MUP^{\otimes n+1}$. In that case we might hope that there results a cosimplical $MUP$-module whose totalization formed in $\Sp$ gives a well-defined stable homotopy type, inspired by the Adams-Novikov type descent argument.
\end{remark}

The proof of~\cref{main} is a straightforward application of the universal property of Thom spectra.
We shall also include a (genuine) equivariant generalization of~\cref{main} in this paper.
See~\cref{equivariant} for details.

\subsection{Acknowledgements}

This paper was originally a part of the Master's thesis written by the author.

The author would like to acknowledge Prof.~Mikio Furuta for suggesting this research direction.
He also wishes to thank Jin Miyazawa for introducing the background issue lying in the Seiberg-Witten Floer homotopy theory~\cite{b1=0} to the author, which stimulated this work.

\subsection{Conventions}

By an \textit{$\infty$-category}, we implicitly mean a quasicategory, which is extensively studied in the literature.
For an $\infty$-category $\Cscr$, the associated mapping spaces will be denoted by $\Map_{\Cscr}(\;,\;)$.
By a \textit{space}, we mean an essentially small Kan complex.
In other words, a space would be constructed as a possibly large $\infty$-groupoid that has only small sets of isomorphism classes of objects as well as the higher simplices.
The $\infty$-categories $\Sscr$ of spaces, $\Cat$ of small $\infty$-categories, $\lCat$ of possibly large $\infty$-categories, $\Prl$ of presentable $\infty$-categories and left adjoint functors between them, and $\Sp$ of spectra will be frequently used. 

The $\infty$-category $\Prl$ is equipped with a symmetric monoidal structure whose tensor product $\otimes$ is left adjoint to the internal hom category $\lFun(\;,\;)$ of colimit-preserving functors.
It is uniquely characterized by the property that the functor $\Fun\left( (-)^\op , \Sscr\right)\colon \Cat \to \Prl$ becomes symmetric monoidal.
We say a presentable $\infty$-category is \textit{presentably $\Ebb_n$-monoidal} if it is equipped with an $\Ebb_n$-algebra structure in $\Prl$.

For an $\infty$-operad $\Oscr$ and $\Oscr$-monoidal $\infty$-categories $\Cscr$ and $\Dscr$, the $\infty$-category of lax $\Oscr$-monoidal functors from $\Cscr$ to $\Dscr$ is denoted by $\Alg_{\Cscr/\Oscr}(\Dscr)$, also by $\Alg_{\Cscr}(\Dscr)$ when $\Oscr$ is the terminal $\infty$-operad $\Ebb_\infty$, and yet by $\Alg_{/\Oscr}(\Dscr)$ when $\Cscr=\Oscr$.
We write it $\Alg_{\Oscr}(\Dscr)$ instead of $\mathsf{Mon}_{\Oscr}(\Dscr)$ even if the symmetric monoidal structure is given by cartesian products in $\Dscr$.
We also write $\calg(\Dscr)$ for $\Alg_{\Ebb_\infty}(\Dscr)$.

\begin{remark}
    When $X$ is an $\infty$-groupoid, we canonically identify the opposite $\infty$-category $X^\op$ with $X$ itself.
    This is achieved for example by the functorial zig-zag \[X^\op \xleftarrow{\sim} \mathrm{Tw}(X) \xrightarrow{\sim} X\] provided by the twisted arrow $\infty$-category construction.
\end{remark}

\section{Recollections of twisted parametrized spectra}

\subsection{Picard $\infty$-groupoids of presentably monoidal $\infty$-categories}

Let $\Rscr$ be a presentably $\Ebb_n$-monoidal $\infty$-category.
The full subcategory $\Rscr^\times$ spanned by the invertible objects is naturally endowed with an $\Ebb_n$-monoidal structure
.

\begin{definition}\label{def:pic}

    The Picard $\infty$-groupoid of $\Rscr$ is defined to be the maximal $\infty$-subgroupoid of the $\infty$-category of invertible objects. 
    \[
        \Pic(\Rscr) \coloneqq (\Rscr^\times){}^\simeq = (\Rscr^\simeq){}^\times
    \]
\end{definition}

By construction, there is a pullback square of $\infty$-categories 
\[
    \begin{tikzcd}
        \Aut_\Rscr(\unit) \arrow[r] \arrow[d] \arrow[rd, "\lrcorner"{anchor=center, pos=0.125}, draw=none] & \left(\Pic\Rscr\right){}_{/\unit} \arrow[d, twoheadrightarrow] \\
        \left\{\unit\right\} \arrow[r] & \Pic(\Rscr)
    \end{tikzcd}
\]
where $\unit$ is the unit object. 
Since the slice $(\Pic\Rscr){}_{/\unit}$ of an $\infty$-groupoid is contractible and since \(\Pic\Rscr\) is grouplike, there is an equivalence of (the underlying) spaces \[\left(\pi_0 \Pic\Rscr\right) \times \B{\Aut}_\Rscr(\unit) \simeq \Pic\Rscr.\]

In~\cite{ABG}, it is shown that the possibly large $\infty$-groupoid $\Pic(\Rscr)$ is essentially small, and consequently the Picard $\infty$-groupoid functor together with the presheaf functor $\Pre \colon \Sscr \to \Prl$ gives rise to the following adjunction.
\begin{equation} \label{adj}
    \begin{tikzcd}
        \Pre\, \colon
        \Alg_{\Ebb_n}^{\grouplike} (\Sscr) \arrow[r, bend left=25, "{}"{name=F}] & \Alg_{\Ebb_n}(\Prl) \,:\Pic
        \arrow[l, bend left=25, "{}"{name=G}]
        \arrow[phantom, from=F, to=G, "\dashv" rotate=-90]
    \end{tikzcd}
\end{equation}
Thus, for a presentable $\Ebb_n$-monoidal $\infty$-category $\Rscr$, the adjunction counit \begin{equation}\Sscr_{/\Pic(\Rscr)} \to \Rscr\label{counit}\end{equation} is an $\Ebb_n$-monoidal functor that preserves colimits. This functor is uniquely characterized by the colimit-preserving property and the property that when restricted to $\Pic_R$ via the Yoneda embedding $\Pic(\Rscr)\to \Sscr_{/\Pic(\Rscr)}$, it agrees with the canonical inclusion $\Pic(\Rscr) \to \Rscr$.

Let $R$ be an $\Ebb_{n+1}$-ring spectrum. The $\infty$-category $\Rscr=\Mod_R$ of right $R$-module spectra is a presentably $\Ebb_n$-monoidal stable $\infty$-category~\cite{HA}. In this case, write $\Pic_R$ for the Picard $\infty$-groupoid $\Pic(\Mod_R)$. 
Since $R$ is the unit object in $\Mod_R$, the Picard $\infty$-groupoid $\Pic_R = \pi_0\Pic_R \times \B\Aut_\Rscr(\unit)$ is equivalent to $\pi_0\Pic_R \times \B\GL_1R$.
We fix an $\Ebb_{n+1}$-ring $R$ throughout this paper.

\subsection{Parametrized spectra and Thom spectra}

We briefly recall the theory of parametrized objects in an $\infty$-category developed in~\cite{ABG}.

For a presentable $\infty$-category $\Cscr$, there exists a unique limit-preserving functor 
\[\Cscr_{/(-)}\colon \Sscr^\op \to \Prl\] that carries a point $\ast$ to $\Cscr$, due to the free cocompletion universality of $\Sscr$.
The property of preservation of limits determines the value at an arbitrary space $X \simeq \displaystyle\lim_{X} \ast$, namely $\Cscr_{/X} \simeq \Fun(X,\Cscr)$. 
The $\Cscr$-valued presheaf category $\Cscr_{/X}$ is referred to as the $\infty$-category of parametrized objects parametrized by $X$.
When $\Cscr=\Sscr$, the slice category $\Sscr_{/X}$ coincides with $\Fun(X, \Sscr)$ basically by the straightening-unstraightening theorem of~\cite{HTT}. Moreover, the equivalence commutes with base-change: 
\[
    \begin{tikzcd}
    \Fun(X,\Sscr) \arrow[d, "{f_!}"'] \arrow[r, "\sim"] & \Sscr_{/X} \arrow[d, "{f_!}"] \\
    \Fun(Y,\Sscr) \arrow[r, "\sim"] & \Sscr_{/Y}
    \end{tikzcd}
\]
It follows, in particular, that the object $p\colon E \to X$ in $\Sscr_{/X}$ corresponding to a functor $F \colon X \to \Sscr$ has total space equivalent to $\displaystyle\colim_{X} F$.

Let $f\colon X \to \Pic_R$ be a map of $\infty$-groupoids, which is considered as a family of invertible $R$-modules parametrized over a space $X$. The Thom spectrum associated to the map $f$ is defined as the colimit of the following composite functor. \[Mf \coloneqq \colim(X \xrightarrow{f} \Pic_R \subset \Mod_R) = p_!f\]
Here, $p$ is the map $p\colon X \to \ast$.
Since generally colimits commute and $M$ sends each point $\ast \to \Pic_R$ to itself in $\Mod_R$, the adjunction counit $\Sscr_{/\Pic_R} \to \Mod_R$ in~\eqref{counit} gives the Thom spectrum functor \[M \colon \Sscr_{/\Pic_R} \to \Mod_R.\]
Moreover, the monoidal property of the adjunction~\eqref{adj} implies that the Thom spectrum of an $\Ebb_n$-monoidal map $f\colon X \to \Pic_R$ is an $\Ebb_n$-ring. \[M \colon \Alg_{/\Ebb_n}(\Sscr_{/\Pic_R}) \to \Alg_{/\Ebb_n} (\Mod_R)\]
The monoidal property of Thom spectra can also be justified by the following adjunction provided by~\cite[Corollary 2.11]{simple}.
\begin{equation}\begin{tikzcd}\label{oplke}
    \Alg_{X/\Ebb_n} (\Mod_R) 
    \arrow[r, bend left=25, "{p_!}"'{name=F}] & \Alg_{/\Ebb_n} (\Mod_R)
    \arrow[l, bend left=25, "{p^\ast}"'{name=G}]
    \arrow[phantom, from=F, to=G, "\dashv" rotate=-90]
\end{tikzcd}\end{equation}

\begin{definition}
    Let $X$ be an $\Ebb_n$-space and suppose that a map $f\colon X \to \Pic_R$ is an $\Ebb_n$-monoidal map.
    We say a map $Mf \to R$ is an \textit{$\Ebb_n$-orientation} when it has a structure of a map of $\Ebb_n$-rings.
\end{definition}

The following lemma is well-known. We will record an argument for the readers convenience.

\begin{lemma}\label{lemma:nullhomotopy}
    Let $f\colon X \to \Pic_R$ be an $n$-fold loop map. Then there is a natural equivalence
    \[\Mapl_{\Alg_{/\Ebb_n}(\Mod_R)} (Mf, R) \; \simeq \; \left\{\B^n f\right\} \quad \underset{\strut\mathclap{\Map_{\Sscr_{\!\ast}}\!(\B^n\! X, \B^n\Pic_R)}}{\times} \quad \{\ast\}. \]
    In other words, an $\Ebb_n$-orientation canonically corresponds to a trivialization of the $n$-fold delooping of $f$.
\end{lemma}

\begin{proof}
    Combining~\cite[Theorem 3.5]{simple} with~\cite[Lemma 3.15]{simple} gives the following equivalence.
    \[
        \Mapl_{\Alg_{/\Ebb_n}(\Mod_R)} (Mf, R) \; \simeq \;
        \{f\} \quad \underset{\strut\mathclap{\displaystyle\Mapl_{\Alg_{\Ebb_n}\!\Sscr}(X, \Pic_R)}}{\times} \quad \Map_{\Alg_{\Ebb_n}\!\Sscr} (X, (\Pic_R)_{/R})
    \]
    Since the space $(\Pic_R){}_{/R}$ is contractible and there is the looping-delooping equivalence
    \[\Sscr_\ast^{\ge n} \xleftarrow[\sim]{\B^n} \Alg_{\Ebb_n}^\grouplike (\Sscr) \rightleftarrows \Alg_{\Ebb_n} (\Sscr),\]
    the claim follows.
\end{proof}

\begin{example}
    Let $X$ be the infinite loop space $\Zbb\times BU$ corresponding to the connective $K$-theory spectrum $ku$. The complex $J$-homomorphism \[J_\Cbb \colon \Omega^\infty ku \to \Pic_\S\] is an infinite loop map, and thus the associated Thom spectrum \[MJ_\Cbb \simeq MUP = \bigvee_{n=-\infty}^\infty \Sigma^{2n} MU\] is an $\Ebb_\infty$-ring.
\end{example}

The following will be used in the main theorem.

\begin{corollary}\label{corollary:E2MUP}
    Let \(R\) be an even-periodic \(\Ebb_2\)-ring spectrum equipped with an \(\Ebb_2\)-ring map \(MU \to R\).
    Then there exists an \(\Ebb_2\)-ring map \(MUP \to R\).
\end{corollary}
\begin{proof}
    The assumption provides an \(\Ebb_2\)-trivialization of the map \(BU \to \Pic_\S \to \Pic_R\), where \(BU\) is the unit component of \(\Omega^\infty ku\simeq \Zbb \times BU\). We observe that the splitting \(\Omega^\infty ku \simeq \Zbb \times BU\) can be made \(\Ebb_2\): In fact, we can provide an \(\Ebb_2\)-splitting \(\Zbb \to \Omega^\infty ku\) by the computation of \(ku^{2}(\Cbb P^\infty)=ku^2(K(\Zbb,2))\cong \Zbb\).

    To finish the proof, it suffices to prove the \(\Ebb_2\)-triviality of the map \(J_\Cbb \colon \Omega^\infty ku \to \Pic_R\).
    For this, since the splitting \(\Omega^\infty ku \simeq \Zbb \oplus BU\) can be made \(\Ebb_2\) and since the map \(BU\to \Pic_R\) is already \(\Ebb_2\)-trivial, it suffices to show that the map \(J_\Cbb \colon \Zbb = (\Omega^\infty ku)/\!/_{BU} \to \Pic_R\) is \(\Ebb_2\)-trivial.
    Since \(R\) is \(2\)-periodic, the map factors as \(J_\Cbb\colon \Zbb \to B{\GL_1}R \to \Pic_R\). We now see that the mapping space \(\Map_{\Ebb_2}(\Zbb, B{\GL}_1R) \simeq \Map_\ast (\Cbb{P}^\infty, B^3{\GL}_1R)\) is connected since the cohomology of \(\Cbb{P}^\infty\) is even and the homotopy groups of \(B^3{\GL}_1R\) is odd.
\end{proof}


\begin{definition}
    Let $f \colon X \to \Mod_R$ be a parametrized $R$-module on a space $X$. The associated twisted $R$-cohomology $R_f^\bullet(X)$ is defined to be the negative stable homotopy group of the internal hom object $\pi_{-\bullet}F_R(Mf,R)$ in $\Mod_R$.
\end{definition}

\begin{remark}
    By the defining universal property of the Thom spectrum and the projection formula, there is naturally an isomorphism \[R^\bullet_f(X) = \pi_{-\bullet} F_R(Mf,R) \cong \pi_{-\bullet} p_\ast F_{(\Mod_R){}_{/X}}(f,p^\ast R) = \pi_{-\bullet} p_\ast F_X (f,p^\ast R).\]
\end{remark}

\subsection{Twisted parametrized spectra}

We continue to fix an $\Ebb_{n+1}$-ring $R$.
The following definition of $R$\textit{-haunts} and \textit{twisted parametrized spectra}, or $R$\textit{-specters}, is originally due to Douglas~\cite{Douglas} and reformulated in~\cite{ABG}.

\begin{definition}
    The Brauer $\infty$-groupoid $\Br_R$ is defined as the Picard $\infty$-groupoid of the $\Ebb_{n-1}$-monoidal $\infty$-category $\Mod_{\Mod_R}(\Prl{}^{, \omega})$ of $\Mod_R$-modules in the $\infty$-category of compactly-generated presentable $\infty$-categories $\Prl{}^{, \omega}$.
\end{definition}

This Brauer $\infty$-groupoid deloops the Picard $\infty$-groupoid:

\begin{lemma}
    $\Omega\Br_R \simeq \Pic_R$.
\end{lemma}

\begin{proof}
    $\Omega\Br_R \simeq \Aut_{\Mod_{\Mod_R}}(\Mod_R) \simeq (\Mod_R^\times){}^\simeq = \Pic_R$.
\end{proof}

In~\cite{ABG} the twisted parametrized spectra are defined as follows.

\begin{definition}
    An $R$-haunt over a space $X$ is an object in the $\infty$-category of $\Mod_R$-torsors: \[\Fun(X, \Br_R) = \Fun(X, \Pic_{\Mod_R})\]
    
    Given an $R$-haunt $\Hcal \colon X \to \Br_R$, the $\infty$-category of twisted parametrized spectra is the presentable stable $\infty$-category $\displaystyle\lim_{X}\Hcal$.
\end{definition}

Let $\B\Pic_R$ denote the unit component of the Brauer $\infty$-groupoid $\Br_R$.

\begin{example}
    According to~\cite{Brauer}, the Brauer group $\pi_0 \Br$ of the sphere spectrum $\S$ is trivial. Hence it follows that $\Br_\S = \B\Pic_\S$.
\end{example}

In particular, we can restrict our attention to those $R$-haunts especially coming from maps to $\B\Pic_R$, as long as we are interested in twisted parametrized spectra over the sphere spectrum.

\begin{remark}
We demonstrate a concrete model for the $\Pic_R$-bundles.
Since $\Pic_R$ is a grouplike $\Abb_\infty$-space, the delooping $\B\Pic_R$ can be the subcategory of $\hat{\Cat}$ having a unique object $\Mod_R$ that deloops $\Pic_R$.
Therefore, the universal $(\Mod_R)$-bundle $\Escr\Pic_R$ over $\B\Pic_R$ is the following base change of the universal cocartesian fibration.

\[\begin{tikzcd}
    \Escr\Pic_R & \hat{\Cat}\times_{N^{\mathrm{hc}}(\sset)} N^{\mathrm{hc}}(\sset)_{\ast/} \\
    \B\Pic_R & \hat{\Cat}
    \arrow[from=1-1, to=1-2]
	\arrow[two heads, from=1-2, to=2-2]
	\arrow[from=1-1, to=2-1]
	\arrow[hook, from=2-1, to=2-2]
	\arrow["\lrcorner"{anchor=center, pos=0.125}, draw=none, from=1-1, to=2-2]
\end{tikzcd}\]

In particular, the sets of objects, 1-morphisms, and 2-simplices are written as follows.

\begin{gather*}
    \left\{\triangle^0 \to \Escr\Pic_R\right\} = \left\{A\in \Mod_R\right\} \\
    \{A\} \underset{\left\{\triangle^0 \to \Escr\Pic_R\right\}}{\times}
    \left\{\triangle^1 \to \Escr\Pic_R\right\} \underset{\left\{\triangle^0 \to \Escr\Pic_R\right\}}{\times} \{B\} \hspace{90pt}  \\ \hspace{90pt} = \left\{f\colon A \underset{R}{\otimes} D \to B \setmid 
    D \text{ is invertible in }\Mod_R\right\}
\end{gather*}

\[
    \left\{\triangle^2 \to \Escr\Pic_R\right\} = \left\{\begin{tikzcd}
    {A \underset{R}{\otimes} D \underset{R}{\otimes} F} && C \\
	& {B \underset{R}{\otimes} F}
	\arrow[""{name=0, anchor=center, inner sep=0}, "h", from=1-1, to=1-3]
	\arrow["g"', from=2-2, to=1-3]
	\arrow["{f\otimes F}"', from=1-1, to=2-2]
	\arrow["\simeq" sloped, shift right=2, shorten <=3pt, phantom, from=0, to=2-2]\end{tikzcd}
    \middle| D,\, F \text{ are invertible.}\right\}
\]
\end{remark}

For an $\S$-haunt $\Hcal \colon X \to \Br_\S$, we write $\Hcal_R$ for the following composite. \[X \to \Br_\S \to \Br_R\]
It is basically the delooping of the functor $(\bullet)\otimes R\colon \Pic_\S \to \Sp \to \Mod_R$.

We are interested in the conditions in which the twisted parametrized spectra in $\lim \Hcal$ have twisted $R$-cohomology theory. In the next section, we provide a sufficient condition for the existence of twisted $R$-cohomology.

\section{The main results}

The main theorem is stated as follows. Here, a homotopy commutative ring spectrum $R$ is said to be \textit{even-periodic} if $\pi_{\mathrm{odd}} R =0$ and if there exists an element $\beta\in \pi_2 R$ which is a unit in the commutative ring $\pi_\bullet R$.

\begin{theorem}
    Let $R$ be an even-periodic $\Ebb_2$-ring spectrum. Then the composite \[ \Omega^{\infty-1}ku \xrightarrow{\B J_\Cbb} \B\Pic_{\S} \to \B\Pic_R \] is null-homotopic. Moreover, the desired null-homotopy is determined essentially uniquely by the choice of an $\Ebb_1$-complex orientation of $R$.
\end{theorem}

\begin{proof}
    Any even-periodic ring spectrum admits a complex orientation.
    Since the $\Ebb_2$-ring spectrum $R$ is assumed to be even, the theorem of~\cite{genera} implies that there is a lift of the orientation by an $\Ebb_2$-ring map $MU \to R$.
    By periodicity, there is an $\Ebb_2$-ring map $MJ_\Cbb =MUP \to R$.\footnote{This is due to \cref{corollary:E2MUP}.} By \cref{lemma:nullhomotopy}, the $\Ebb_1$-map $MJ_\Cbb \to R$ uniquely determines a null-homotopy of 
    the map $\B J_\Cbb \colon \Omega^{\infty-1}ku \to \B\Pic_R$.
\end{proof}

\begin{corollary}
    Let $R$ be as above and $\Hcal \colon X \to \B\Pic_{\S}$ be a haunt. Suppose that $\Hcal$ factors through the map $\B J_\Cbb \colon \Omega^{\infty-1}ku \to \B\Pic_{\S}$. Then an $\Ebb_1$-complex orientation of $R$ gives an equivalence of $\infty$-categories \[ \lim_{X} \Hcal_R \simeq \lim_{X} \Mod_R \simeq \left( \Mod_R \right)\hbox{}_{/X}. \]
    Consequently, the twisted $R$-cohomology theory for twisted parametrized spectra in $\displaystyle\lim_{X} \Hcal$ is well-defined.
\end{corollary}

\section{Generalization to the equivariant setting}\label{equivariant}

Let $G$ be a compact Lie group and $\Orb_G$ denote its topological orbit category, whose objects are the orbits $G/H$ for $H\leq G$ closed and whose mapping spaces are the spaces of equivariant maps. We work in the $\infty$-category theory internal to the presheaf $\infty$-topos $\Sscr^G\coloneqq\Pre(\Orb_G)$, the $\infty$-category of genuine equivariant $G$-spaces by the theorem of Elmendorf.
The symmetric monoidal $\infty$-category $\Sp^G$ of genuine $G$-spectra is characterized by the property that it universally inverts the finite-dimensional representation spheres in $\Sscr^G_\ast$.
Write $\S_G$ for the $G$-equivariant sphere spectrum, which is a monoidal unit in $\Sp^G$.

\subsection{Presentable $\infty$-categories internal to $G$-spaces}

An $\infty$-category internal to $\Sscr^G$ is simply called a \textit{$G$-category}.
Given a $G$-category $\Cscr$, viewed as a $\lCat$-valued presheaf on $\Orb_G$, its value at $G/H\in \Orb_G$ will be denoted by $\Cscr^H$.
This notation agrees with the fact that the $G$-category $\base$ associated to $\Sscr^G$ is described as the following functor.
\[\base \colon \Orb_G \ni G/H \mapsto \Sscr^H \simeq \Sscr^G_{/(G\!/\!H)} \in \lCat \]
Similarly, the $\infty$-category $\Sp^G$ can be equipped with the $G$-category structure 
\[\underline{\Sp^G} \colon G/H \mapsto \Sp^H \simeq \Sp^G \otimes_{\Sscr^G} \Sscr^H \]
as in~\cite[Definition 4.2]{TwAmbi}, which is in fact presentably symmetric monoidal in the $G$-categorical sense.
We follow~\cite{presentable} the definitions of presentable $G$-categories and $\infty$-operads internal to $\Sscr^G$.
Note that the notions of $G$-operads and symmetric monoidal $G$-categories are straightforward in the sense that they are just the presheaves over $\Orb_G$ valued in the $\infty$-categories of $\infty$-operads and of symmetric monoidal $\infty$-categories, respectively.
Note that our notion of $G$-operads may differ from something called $G$-$\infty$-operads used in equivariant homotopy theory for finite groups.

The $\infty$-category of $G$-functors between $G$-categories $\Cscr$ and $\Dscr$ is described as the end
\[\Fun^G(\Cscr, \Dscr) \simeq \coend{G/H\in \Orb_G} \Fun(\Cscr^H, \Dscr^H), \]
which is the value of the internal functor $G$-category $\Fun(\Cscr, \Dscr)$ evaluated at the terminal orbit $G/G$.
Here, an end is defined as a limit over the suitable twisted arrow $\infty$-category.
We define internal functor category of two $G$-operads $\Cscr$ and $\Dscr$ to be the following end.
\[\Alg^G_\Cscr(\Dscr) \simeq \coend{G/H} \Alg_{\Cscr^H} (\Dscr^H) \]
For example, $\calg^G(\Dscr)$ is equivalent to $\lim_{\Orb_G^\op} \calg(\Dscr^H) \simeq \calg(\Dscr^G)$ since $G/G\in \Orb_G$ is a terminal object.

Let $\PrlG$ denote the $\infty$-category of presentable $G$-categories and functors preserving $G$-colimits between them.
Since the functor $\Sscr^G_{/\bullet}\colon \Sscr^G \to \Mod_{\Sscr^G}(\Prl)$ is symmetric monoidal as well as the functor~\cite[8.3]{presentable} $(\bullet)\otimes_{\Sscr^G} \base$ $\colon \Mod_{\Sscr^G} (\Prl) \to \PrlG$, the composition defines a functor 
\[\Pre_\Omega \colon \calg^\grouplike (\Sscr^G)\simeq \Fun(\Orb_G^\op, \calg^\grouplike(\Sscr)) \rightarrow \calg(\PrlG)\] with right adjoint $\Pic\colon \calg(\PrlG) \to \calg^\grouplike(\Sscr^G)$.
Its value at $\Cscr \in \calg(\PrlG)$ has underlying $G$-category the composite
\[\Pic(\Cscr) \colon \Orb_G^\op \xrightarrow{\Cscr} \calg(\Prl) \xrightarrow{\Pic} \calg^\grouplike(\Sscr^G). \]
The first arrow here exists since $Cscr$ is presentably symmetric monoidal in the $G$-categorical sense.
For a commutative ring $G$-spectrum $R\in \calg(\Sp^G)$, we write $\Pic_R$ for $\Pic\left(\Mod_R(\underline{\Sp^G})\right)$. See~\cite[7.2]{presentable} for the (presentably) symmetric monoidal $G$-category of $R$-modules.

As a corollary, the adjunction counit in particular defines a functor of $\infty$-categories
\[M\colon \Sscr^G_{/\Pic_R} \to \Mod_R \coloneqq \Mod_R(\Sp^G) \]
which is symmetric monoidal, colimit-preserving, and sends a point in $\Pic_R^H$ identified with an invertible object $E$ in $\Mod_{R}(\Sp^H)$ to the induction $\Ind^G_H E$ in $\Mod_{R}(\Sp^G)$.

Taking ends of the both hand side in~\eqref{oplke}, we have an adjunction 
\[p_! \colon \Alg_{X}^G (\Mod_R(\underline{\Sp^G})) \rightleftarrows \calg(\Mod_R(\Sp^G)) \,\colon p^\ast \]
for each $X\in \calg(\Sscr^G)$.
By the above construction of the Thom $G$-spectrum functor, $Mf$ agrees with $p_!f$ for a map $f\in \calg(\Sscr^G_{/\Pic_R})$, also regarded as a functor $X\xrightarrow[f]{} \Pic_R \to \Mod_R(\underline{\Sp^G})$.

\subsection{Generalization of the main theorem}

In order to generalize the main result, we need a $G$-equivariant model of periodic complex Thom spectrum.

Let $\Vect^H$ be the topological category of finite-dimensional orthogonal $H$-representations and equivariant isometries between them. They assemble into a single $G$-category $\underline{\Vect^G} \colon G/H \mapsto \Vect^H$.
We equip it with the symmetric monoidal structure under the direct product of orthogonal representations.
The $G$-functor sending an orthogonal representation to its one-point compactification is a symmetric monoidal $G$-functor $\underline{\Vect^G} \to \underline{\Sp^G}$.
Since every representation sphere is invertible as an equivariant spectrum, it defines a symmetric monoidal map of $G$-spaces $(\underline{\Vect^G}\,)^\simeq \to \Pic_{\S_G}$.
Taking group completion on the left hand side, the $G$-map adjuncts to a $G$-map of infinite loop $G$-spaces, which we call $J_G$.

\begin{theorem}
    Let $\Hcal\colon X \to \B\Pic_{\S_G}$ be a $G$-map.
    Suppose that $\Hcal$ factors through the map $BJ_G$.
    Then there is an equivalence of $\infty$-categories.
    \[\lim\left(X \xrightarrow{\Hcal\otimes M\!J_G} \B\Pic_{M\!J_G} \hookrightarrow \lCatG\right) \simeq \lim_{G\!/\!H\in (\Orb_G)_{/X}^\op}\Mod_{M\!J_G}({\Sp^H}) = (\Mod_{M\!J_G})_{/X} \]
\end{theorem}

\begin{proof}
    It suffices to show that the composite $\xrightarrow{\B J_G} \Pic_{\S_G} \xrightarrow{\otimes M\!J_G}$ is null-homotopic.
    So we replace $X$ by the domain of $J_G$, which is an inifnite loop object. We have a following chain of equivalences.
    \begin{align*}
        \Map_{\calg\Mod_{M\!J_G}} (M\!J_G, M\!J_G) &\simeq \Mapl_{\Alg^G_X \Mod_{M\!J_G}} (J_G, p^\ast M\!J_G) \\
        &\simeq \coend{G/H} \Mapl_{\Alg_{X^{\!H}} \Mod_{M\!J_G}(\Sp^H)} ((J_G)^H, (M\!J_G)^H \circ p) \\
        &\simeq \coend{G/H} \{J_G {}^H\} \underset{\Alg_{X^{\!H}}\Mod_{M\!J_G}(\Sp^H)}{\times} \Alg_{X^{\!H}} \left(\Mod_{M\!J_G}(\Sp^H)_{/M\!J_G{}^H}\right) \\
        &\simeq \coend{G/H} \{J_G{}^H\} \underset{\Mapl_{\calg (\Sscr)}(X^H, \Pic_{M\!J_G}^H)}{\times} \ast \\
        &\simeq \{J_G\} \underset{\Mapl_{\calg\Sscr^{\!G}}(X, \Pic_{M\!J_G})}{\times} \ast \\
        &\to \{BJ_G\} \underset{\Map_{\Sscr_{\!\ast}}(\B X, \B\Pic_{M\!J_G})}{\times} \ast
    \end{align*}
    The first equivalence is by $p_!\vdash p^\ast$, the third equivalence is~\cite[Lemma 2.12]{simple}, the fourth is completely parallel to~\cite[Theorem 3.5]{simple}.
    The space on the left has a base point, namely the identity map, and it defines a desired null-homotopy that belongs to the right hand side.
\end{proof}

\bibliographystyle{unsrt}
\bibliography{bpic}

\end{document}